\documentclass[a4paper,10pt]{article}	
\usepackage[latin1]{inputenc}		
\usepackage{amsmath}						
\usepackage{amssymb}						%
\usepackage{amsthm}						%
\usepackage{amscd}						%
\usepackage{multicol}
\usepackage{makeidx}
\usepackage{graphicx}

\usepackage{hyperref}
\hypersetup{urlcolor=cyan}

\makeindex
\newcommand{\ve}{\mathbf{e}}
\newcommand{\vf}{\mathbf{f}}
\newcommand{\vu}{\mathbf{u}}
\newcommand{\vv}{\mathbf{v}}
\newcommand{\vw}{\mathbf{w}}
\newcommand{\vx}{\mathbf{x}}
\newcommand{\vy}{\mathbf{y}}

\newcommand{\set}[1]{\left\{#1\right\}}
\newcommand{\genset}[2]{\left\{\,#1{\,\, : \,\,}#2\,\right\}}
\newcommand{\setC}{{\mathbf{C}}}
\newcommand{\setN}{{\mathbf{N}}}

\newcommand{\setZ}{{\mathbf{Z}}}

\newcommand{\id}{\mathrm{id}}

\newcommand{\iso}{\cong}
\newcommand{\inv}{^{-1}}
\newcommand{\Aut}{\ensuremath{\mathrm{Aut}}}
\newcommand{\istmit}[1]{\overset{#1}{=}}
\newcommand{\folgtmit}[1]{\overset{#1}{\Longrightarrow}}
\newcommand{\scp}[2]{\ensuremath\left\langle #1 \,, #2 \right \rangle} 
\newcommand{\gen}[1]{\left\langle #1 \right\rangle} 
\newcommand{\trace}{\ensuremath{\mathrm{Tr}}}%

\newcommand{\mcb}{\mathcal{B}}

\swapnumbers
\newtheorem{theorem}{Theorem}[section]

\newtheorem{proposition}[theorem]{Proposition}

\theoremstyle{definition}
\newtheorem{definition}[theorem]{Definition}

\theoremstyle{remark}
\newtheorem{remark}[theorem]{Remark}

\newcommand{\act}{.}
\newcommand{\sumn}[2]{\ensuremath \sum_{#1 = 1}^n \, #2 }
\newcommand{\sumg}[2]{\ensuremath \frac{1}{|G|}\sum_{#1 \in G} \, #2 }
\def\lri/{ljum\-mig  ring}		
\def\irr/{irreducible}
\newcommand{\kwd}[1]{\emph{#1}\index{#1}} 
\newcommand{\inx}[1]{#1\index{#1}} 

\newcommand{\spec}{\ensuremath \mathrm{spec}}
\newcommand{\diag}{\ensuremath \mathrm{diag}}

\title{A Gentle Introduction to a Beautiful Theorem of Molien}
\author{Holger Schellwat\\\\
\href{mailto:holger.schellwat@oru.se}{\texttt{holger.schellwat@oru.se}},
Örebro universitet, Sweden\\
Universidade Eduardo Mondlane, Mo\c{c}ambique\\
}
\date{12 January, 2017}
 
\begin{document}
\maketitle

\begin{abstract}
The purpose of this note is to give an accessible proof of
Moliens Theorem in Invariant Theory, in the language of
today's Linear Algebra and Group Theory, in order to 
prevent this beautiful theorem from being forgotten.
\end{abstract}

\tableofcontents

\newpage
\section*{Introduction}\label{sintro}	

We present some memories of a visit to the ring zoo in 2004. This time 
we met an animal looking like a unicorn, known by the name of
invariant theory. It is rare, old, and very beautiful. The purpose
of this note is to give an almost self contained introduction to
and clarify the proof of the amazing theorem of Molien, as 
presented in \cite{Sloane1}. An  introduction into
this area, and much more, is contained in \cite{Sturmfels}.
There are many very short proofs of this theorem, for instance in 
\cite{Stanley}, \cite{Hu2}, and \cite{Tambour}. \par

Informally, Moliens Theorem is a power series generating function formula for
counting the dimensions of subrings of homogeneous polynomials 
of certain degree which are invariant under the action of a finite
group acting on the variables. 
As an apetizer, we display this stunning formula:
$$
\Phi_G(\lambda) := \frac{1}{|G|} \sum_{g\in G} \frac{1}{\det(\id - \lambda T_g)}  
$$ 
We can immediately see elements of linear algebra, representation theory, and 
enumerative combinatorics in it, all linked together.
The paper \cite{Sloane1} nicely shows  how
this method can be applied in Coding theory. For 
Coding Theory in general, see \cite{jubi}.\par

Before we can formulate the Theorem, we need to set the 
stage by looking at some Linear Algebra (see \cite{Roman}), 
Group Theory (see \cite{Hu}), and Representation Theory (see \cite{Sagan} and \cite{Tambour}).
\newpage 
\section{Preliminaries}\label{spreliminaries}

Let $V\iso \setC^n$ be a finite dimensional complex inner product  space with
orthonormal basis $\mcb = (\ve_1,\dots,\ve_n)$ and let $\vx = (x_1,\dots,x_n)$ 
be the orthonormal basis of the algebraic dual space $V^\ast$ satisfying
$\forall 1\le i,j \le n : x_i(\ve_j) = \delta_{ij}$. Let
$G$ be a finite group acting unitarily linear on $V$ from the left, that is,
for every $g\in G$ the mapping $V \to V, \vv \mapsto g\act \vv$ is 
a unitary bijective linear transformation. Using coordinates, this can be expressed as
$[g\act \vv]_\mcb = [g]_{\mcb,\mcb} [\vv]_\mcb$, where $[g]_{\mcb,\mcb}$ is unitary. 
Thus, the action is a unitary representation of $G$, or in other words, a $G$--module.
Note that we are using \inx{left composition} and column vectors, 
i.e. $\vv = (v_1, \dots , v_n) \istmit{convention} [v_1 \, v_2 \, \dots \, v_n]^\top$,
c.~f.~\cite{Anton}.  
\par

The elements of 
$V^\ast$ are \inx{linear forms}(linear functionals), and the elements $x_1,\dots, x_n$, 
looking like variables, are also linear forms, this will be important later. \par
Thinking of $x_1,\dots, x_n$ as variables, we may view (see \cite{Tambour}) $S(V^\ast)$,
the \kwd{symmetric algebra} on $V^\ast$ as the algebra 
$R := \setC[\vx] := \setC[x_1,\dots, x_n] $ of polynomial functions $V\to \setC$
or polynomials in these variables (linear forms). 
It is naturally graded by degree as
$R = \bigoplus_{d \in \setN} R_d$, where $R_d$ is the vector space
spanned by the polynomials of (total) degree $d$, 
in particular,  $R_0 = \setC$, and
$R_1 = V^\ast$.\par

The action of $G$ on $V$ can be lifted to an action on $R$.

\begin{proposition}\label{pindact}
Let $V$, $G$, $R$ as above. 
Then the mapping $\act : G \times R \to R, (g,f) \mapsto g\act f$
defined by $(g\act f) (\vv) := f(g\inv \act \vv)$ for $\vv \in V$
is a left action. 
\end{proposition}
\begin{proof} For $\vv \in V$,  $g,h \in G$, and $f \in R$ we check
\begin{enumerate}
	\item $(1\act f)(\vv) = f(1\inv\act \vv) = f(1\act \vv) = f(\vv)$
	\item \begin{multline*}
			((hg)\act f)(\vv) = f((hg)\inv \act \vv) = 
				f((g\inv h\inv)\act \vv) =\\ f (g\inv \act (h\inv \act \vv))
				= (g\act f)(h\inv \act \vv) = (h\act (g\act f))(\vv)
				\end{multline*}
\end{enumerate}
\end{proof} 
 
In fact, we know more.
\begin{proposition}\label{pindact2}
Let $V$, $G$, $R$ as above. 
For every $g \in G$, the mapping $T_g: R \to R, f \mapsto g\act f$
is an algebra automorphism preserving the grading, i.e.
$g.R_d \subset R_d$ (here we do not bother about surjectivity).
\end{proposition}
\begin{proof} For $\vv \in V$,  $g\in G$,  $c\in \setC$, and $f,f' \in R$ we check
\begin{enumerate}
\item	\begin{multline*}
			(g\act (f+f'))(\vv) =
			(f+f')(g\inv \act \vv) =
			f(g\inv \act \vv)   + f'(g\inv \act \vv) =\\
			(g\act f)(\vv) + (g\act f')(\vv) =
			(g\act f + g\act f')(\vv)
			\textrm {, thus } g\act (f+f') =g\act f + g\act f'
			\end{multline*}
\item	\begin{multline*}
			(g\act (f\cdot f'))(\vv) =
			(f\cdot f')(g\inv \act \vv) =
			f(g\inv \act \vv)   \cdot  f'(g\inv \act \vv) =\\
			(g\act f)(\vv) \cdot  (g\act f')(\vv) =
			(g\act f \cdot  g\act f')(\vv)
			\textrm {, thus } g\act (f\cdot f') =g\act f \cdot  g\act f'
			\end{multline*}
\item $
			(g\act (cf))(\vv) = (cf)(g\inv \act \vv) = c (f(g\inv \act \vv)) =
			c ((g\act f)(\vv)) = (c (g\act f))(\vv) 
			$
\item By part $2.$ it is clear that the grading is preserved.
\item To show that $f \mapsto g\act f$ is bijective it is enough to 
			show that this mapping is injective on the finite dimensional
			homogeneous components
			$R_d$. Let us introduce a name for this mappig, say
			$T_g^d : R_d \to R_d, f \mapsto g\act f$. Now
			$f \in \ker(T_g^d)$  implies that
			$g\act f = 0 \in R_d$, i.e. $g\act f$ is a polynomial mapping from
			$V$ to $\setC$ of degree $d$ vanishing identically,
			$\forall \vv \in V: (g\act f)(\vv) = 0$. By definition of the 
			extended action we have
			$\forall \vv \in V: f(g\inv \act \vv) = 0$.
			Since $G$ acts on $V$ this implies that
			$\forall \vv \in V : f(\vv) = 0$, so $f$ is the zero mapping.
			Since our ground field has characteristic $0$, this implies that
			$f$ is the zero polynomial, which we may view as an element of every $R_d$.
			See for instance \cite{Cox}, proposition 5 in section 1.1. 
\item Note that every $T_g^d$ is also surjective, since all group elements have their inverse in $G$.			
\end{enumerate}
\end{proof} 

Both propositions together give us a homomorphism from $G$ into
$\Aut(R)$. They also clarify the r\^ole of the \emph{induced} 
matrices, which are classical in this area, as mentionend in \cite{Sloane1}.
Since the monomials $x_1,\dots,x_n$ of degree one form a basis for
$R_1$, it follows from the proposition that their products
$\vx_2 := (x_1^2,x_1 x_2,x_1 x_3,\dots,x_1 x_n, x_2^2,x_2 x_3,\dots)$ 
form a basis for $R_2$, and, in general, the monomials of degree $d$
in the linear forms (!) $x_1,\dots,x_n$   form a basis $\vx_d$ of
$R_d$. Clearly, they certainly span $R_d$, and by the last observation
in the last proof they are linearly independent.\par

\begin{definition}\label{dinduced}
In the context from above, that is $g \in G$, $f \in R^d$, and $\vv \in V$,
we define
$$
T_g^d : R_d \to R_d, f \mapsto g\act f : R^d \to \setC, \vv \mapsto f(g\inv \act \vv) = f(T_{g\inv} (\vv))  
.$$
\end{definition}

\begin{remark}\label{rinduced}
In particular, we have $(T_g^1 (f))(\vv) = f(T_{g\inv}(\vv) ),$ see
proposition \ref{pprop0} below.
\end{remark}

Keep in mind that a function $f \in R_d$ maps  to $T_g^d (f) = g\act f$.
Setting $A_g := [T_g^1]_{\vx,\vx}$, then 
$A_g^{[d]} := [T_g^d]_{\vx_d,\vx_d}$ is the $d$--th induced matrix
in \cite{Sloane1}, because $T_g^1(f\cdot f') = T_g^1(f)\cdot T_g^1(f')$.
Also, if $f,f'$ are eigenvectors of $T_g^1$ corresponding to the eigenvalues
$\lambda,\lambda'$, then $f\cdot f'$ is an eigenvector of $T_g^2$
with eigenvalue $\lambda \cdot \lambda'$, because
$T_g(f\cdot f') = T_g(f) \cdot T_g(f') = (\lambda f) \cdot (\lambda' f')
= (\lambda \lambda')(f \cdot f')$. All this generalizes to $d>2$, we 
will get back to that later. \par

We end this section by verifying two little facts needed in the next section.
\begin{proposition}\label{ppropd}
The \kwd{first induced operator} of the inverse of a group element $g\in G$
is given by $T_{g\inv}^1  = (T_g^1)\inv$.
\end{proposition}
\begin{proof}
Since $\dim(V^\ast) < \infty$, it is sufficient to prove
that $T_{g\inv}^1 \circ  T_{g}^1  = \id_{V^\ast}$. 
Keep in mind that $(T_g^1 (f))(\vv) = f (T_{g\inv} (\vv))$. For arbitrary
$f \in V^\ast$ we see that
\begin{align*}
(T_{g\inv}^1 \circ  T_{g}^1 )(f) = T_{g\inv }^1 (  T_{g}^1 (f)) = T_{g\inv}^1 ( g\act f)
   = g\inv \act ( g\act f) = (g\inv g)\act f = f.
\end{align*}
\end{proof}

We will be mixing group action notation and composition
freely, depending on the context. The following observation is a
translation device.

\begin{proposition}\label{pprop0}
For $g \in G$ nd  $f \in V^\ast$ the following holds:
$$
T^1(f) = g\act f = f \circ T_{g\inv}.
$$
\end{proposition}
\begin{proof}
For $\vv \in V$ we see
$(T^1(f))(\vv) =  (g\act f)(\vv) \istmit{def}  f(g\inv \act \vv  ) =  f( T_{g\inv}(\vv)  ).$
\end{proof}

\newpage
\section{The Magic Square}\label{ssquare}
Remember that we require a unitary representation of $G$, 
that is the operators $T_g : V \to V$ need to be unitary, 
i.e. $\forall g \in G : (T_g)\inv = (T_g)^\ast$. 
The first goal of this sections is to show that this implies
that the induced operators $T_g^d : R_d \to R_d, f \mapsto g\act f$
are also unitary. We saw that $T_g^1 = V^\ast$, the algebraic 
dual of $V$. In order to understand the operator duals of 
$V$ and $V^\ast$ we need to look on their inner products first.
We may assume that the operators $T_g$ are unitary with respect to
the standard inner product 
$\scp{\vu}{\vv} = [\vu]_{\mcb, \mcb} \bullet \overline{[\vv]_{\mcb, \mcb}}$,  
where $\bullet$ denotes the dot product.\par

Before we can speak of unitarity of the induced  operators $T_g^d$
we have to make clear which inner product applies on $R^1 = V^\ast$.
Quite naively, for $f,g \in  V^\ast$ we are tempted to define 
$\scp{f}{g} = [f]_{\vx, \vx} \bullet \overline{[g]_{\vx, \vx}}$. \par

We will motivate this in a while, but first we take a look
at the diagram in \cite{Roman}, chapter10, with our objects:

$$
\begin{CD}
\quad @<T_g^\times<< \quad\\
R^1 = V^\ast @>T_g^1>>V^\ast = R^1 \\
@VVPV@VVPV\\
V@>T_g>>V\\
\quad @<T_g^\ast<< \quad\\
\end{CD}
$$

Here $P$ (\lq \lq Rho\rq\rq\ ) denotes the \inx{Riesz map}, see \cite{Roman}, 
Theorem 9.18, where it is called $R$, but $R$ denotes already our big ring. 
We started by looking at the operator $T_g$, which is unitary, so 
its inverse is the Hilbert space adjoint $T_g^\ast$. Omiting the names
of the bases we have $[T_g^\ast] = [T_g]^\ast $. We 
also see the operator adjoint  $T_g^\times$ with matrix
$[T_g^\times] = [T_g]^\top$, the transpose.
However, the arrow for $T_g^1$ is not in the original diagram, but 
soon we will see it there, too. \par

Fortunately, the Riesz map $P$ turns a linear form into a vector
and its inverse $\tau : V \to V^\ast$ maps a vector to 
a linear form, both are conjugate isomorphisms. 
This is mostly all we need in order to show that $T_g^1$ is unitary. 
In the following three propositions we use
that $V$ has the orthonormal basis $\mcb$ and that $V^\ast$ has the orthonormal basis
$\vx$.

\begin{proposition}\label{ppropa}
For every  $f \in V^\ast$ the coordinates of its Riesz 
vector are given by 
$$[P(f)]_\ve = (\overline{f(\ve_1)}, \dots ,  \overline{f(\ve_n)}).$$
\end{proposition}
\begin{proof}
Writing $\tau$ for the inverse of $P$, we need to show that
$$
P(f) = \sumn{i}{\overline{f(\ve_i)}}\ve_i    
$$ 
which is equivalent to
$$
f = \tau \left ( \sumn{i}{\overline{f(\ve_i)}}\ve_i  \right ).
$$
It is sufficient to show the latter for values of $f$ on the basis 
vectors $\ve_j$, $1 \le j \le n$. We obtain
\begin{align*}
\left (\tau \left ( \sumn{i}{\overline{f(\ve_i)}}\ve_i  \right )\right ) (\ve_j) &=
\scp{\ve_j}{ \left ( \sumn{i}{\overline{f(\ve_i)}}\ve_i  \right )}   
=
\sumn{i}{\scp{\ve_j}{ \left ( {\overline{f(\ve_i)}}\ve_i  \right )} }  \\
&= \overline{\overline{f(\ve_i)}} \sumn{i}{\scp{\ve_j}{ \ve_i  } } 
 = f(\ve_i) \cdot 1.
\end{align*}
\end{proof}

In particular, this implies that $P(x_i) = \ve_i$.

\begin{proposition}\label{ppropb}
Our makeshift inner product on $V^\ast$ satisfies
$$
\scp{f}{g} = \scp{P(f)}{P(g)}
,$$
where $f,g \in V^\ast$.
\end{proposition}
\begin{proof}
By our vague definition we have
$\scp{f}{g} = [f]_{\vx, \vx} \bullet \overline{[g]_{\vx, \vx}}$. 
It is enough to show that 
$\scp{x_i}{x_j} = \scp{P(x_i)}{P(x_j)}$. From the comment after the proof of 
Proposition \ref{ppropa} we obtain  
$$\scp{P(x_i)}{P(x_j)} = \scp{\ve_i}{\ve_j} = \delta_{ij} = \ve_i \bullet \ve_j 
=  [x_i]_{\vx, \vx} \bullet \overline{[x_j]_{\vx, \vx}}
.$$
\end{proof}
Hence, our guess for the inner product on $V^\ast$ was correct.
We will now relate the Riesz vector of $f \in V^\ast$ to the
Riesz vector of $f \circ T_g\inv$. Recall that the Riesz vector of $f \in V^\ast$
is the unique vector $\vw = P(f)$ such that $f(\vv) = \scp{\vv}{\vw}$ for 
all $\vv \in V$. If $f \ne 0$ it can be found by scaling any nonzero
vector in the cokernel of $f$, which is one--dimensional, see \cite{Roman},
in particular Theorem 9.18.

\begin{proposition}\label{pprope}
Let $T_g : V \to V$ be unitary, $f \in V^\ast$, $\vw = P(f)$ the  vector of $f \in V^\ast$.
Then $T_g(\vw)$ is the Riesz vector of $f \circ T_g\inv$, i.e. the Riesz vector of $T^1_g(f)$. 
\end{proposition}
\begin{proof}
We may assume that $f \ne 0$.
Using the notation $\gen{\vw}$ for the one--dimensional subspace spanned by $\vw$, 
we start with a little diagram:
$$
\gen{\vw} \odot \ker(f) \overset{T_g}{\longrightarrow} \gen{T_g(\vw)} \odot \ker(f \circ T_g\inv ), 
$$
wheere $\odot$ denotes the orthogonal direct sum.\par

We need to show that $f \circ T_g\inv = \scp{\cdot}{T_g(\vw)}$, i.e.
that $(f \circ T_g\inv)(\vv) = \scp{\vv}{T_g(\vw)}$ for all $\vv \in V$.
Since $\vw = P(f)$ the  vector of $f$, we have $f(\vv) = \scp{\vv}{\vw}$  for all $\vv \in V$.
We obtain
\begin{align*}
(f \circ T_g\inv)(\vv) &= \scp{T_g\inv(\vv)}{\vw} \istmit{T_g \,\,\mathrm{unitary} }
                          \scp{\vv}{T_g(\vw)}. 
\end{align*}
From remark \ref{rinduced} we conclude that $f \circ T_g\inv = T^1_g(f)$. 
\end{proof}

Observe that proposition \ref{pprope} implies the commutativity of the following 
two diagrams.
$$
\begin{CD}
V^\ast @>T_g^1>>V^\ast\\
@VVPV@VVPV\\
V@>T_g>>V\\
\end{CD}
\qquad \mathrm{and } \qquad
\begin{CD}
V^\ast @>(T_g^1)\inv>>V^\ast\\
@VVPV@VVPV\\
V@>(T_g)\inv>>V\\
\end{CD}
$$
Indeed, \ref{pprope} implies 
\begin{align}
P \circ T_g^1 &= T_g \circ P \\ 
P \circ (T_g^1)\inv &= (T_g)\inv \circ P
\end{align}

\begin{proposition}\label{pproplink}
The first induced operator $T_g^1$ is unitary.
\end{proposition}
\begin{proof}
We may use that $T_g$ is unitary, that is, 
$$
\scp{T_g(\vv)}{\vw} = \scp{\vv}{(T_g)\inv(\vw)}
= \scp{\vv}{(T_{g\inv})(\vw)} \qquad (\ast)
.$$
Let $f,h \in V^\ast$ arbitrary, $\vw := P(f)$, and $\vu := P(h)$.
We need to check that $\scp{(T_g^1)(f)}{h} = \scp{f}{(T_g^1)\inv(h)}$.
We see that
\begin{align*}
\scp{(T_g^1)(f)}{h} &\istmit{\mathrm{proposition }\ref{ppropb}} 
        \scp{(P\circ T_g^1)(f)}{P(h)} \istmit{(1)} \scp{(T_g\circ P )(f)}{P(h)} \\
		&= \scp{(T_g( P ))(f)}{P(h)} = \scp{T_g(\vw)}{\vu}	\istmit{\ast} \scp{\vw}{T_g\inv(\vu)}	\\
	  &= \scp{P(f)}{T_g\inv(P(h))} = \scp{P(f)}{(T_g\inv \circ P ) (h)}\\
		&\istmit{(2)} \scp{P(f)}{(  P \circ (T_g^1)\inv) (h)} =  \scp{P(f)}{  P  ((T_g^1)\inv (h))}\\
		&= \scp{f}{(T_g^1)\inv(h)}
\end{align*}
\end{proof}

After having looked at eigenvalues we will see that this generalizes to higher degree,
that $T_g^d$ is diagonalizable for all $d\in \setZ^+$. But first let us look at the matrix version 
of proposition \ref{pproplink}.

\begin{proposition}\label{ppropf}
$$
[T^1_g]_{\vx,\vx} = \overline{[T_g]_{\ve,\ve}}
$$
\end{proposition}
\begin{proof}
Let $A := [T_g]_{\mcb, \mcb} = [A_1| \cdots |A_i| \cdots | A_n] = [a_{i,j}]$ and 
$B := [T_g^1]_{\vx,\vx} = [B_1| \cdots |B_i| \cdots | B_n] = [b_{i,j}]$. 
We will use the commutativity of the diagram, i.e.
$P\inv \circ T_g \circ P = T_g$, which we will mark as $\square$. No, the proof 
is not finished here.
We get $T_g(\ve_i) = A_i = \sumn{k}{a_{k,i}} \ve_k$ and
\begin{align*}
T_g^1 (x_i) &\istmit{\square} (P\inv \circ T_g \circ P)(x_i) = P\inv ( T_g ( P  (x_i)) \\
     &\istmit{\ref{ppropa}} P\inv ( T_g (  \ve_i)) =  P\inv \left ( \sumn{k}{a_{k,i}} \ve_k \right) 
		  \istmit{\textrm{konj.}} \sumn{k}{  \overline{a_{k,i}} P\inv \left ( \ve_k \right) }\\
		 &\istmit{\ref{ppropa}} \sumn{k}{\overline{a_{k,i}} x_k}	
\end{align*}
On the other hand,
$[T^1_g(x_i)]_{\vx} = [T^1_g]_{\vx,\vx} \ve_i = B_i$ implies 
$T^1_g(x_i) = \sumn{k}{b_{k,i}}\ve_k$.
Together we obtain  $b_{k,i} = \overline{a_{k,i}}$, and the proposition follows.
\end{proof}
\newpage
\section{Averaging over the Group}\label{sreynolda}
Now we apply averaging to obtain self-adjoint operators.
\begin{definition}\label{dreynolds} We define the following operators:
\begin{enumerate}
	\item $\displaystyle \hat{T} : V\to V, \vv \mapsto  \hat{T}(\vv) 
	      := \sumg{g}{T_g(\vv)}$ 
	\item $\displaystyle \hat{T^1} : V^\ast\to V^\ast, f \mapsto  \hat{T^1}(f) 
	      := \sumg{g}{T^1_g(f)}$ 
\end{enumerate}

\end{definition}

These are sometimes called the \kwd{Reynolds} operator of $G$.

\begin{proposition}\label{preynolds}
The operators $ \hat{T}$ and $\hat{T^1}$ are self-adjoint (Hermitian).
\end{proposition}
\begin{proof}
The idea of the averaging trick is that if $g\in G$ runs through all group 
element and $g' \in G$ is fixed, then the products  $g'g$ run also through all group elements.
We will make use of the facts that every $T_g$ and every $T^1_g$ is unitary. 
\begin{enumerate} 
	\item We need to show that $\scp{\hat{T}(\vv)}{\vw} = \scp{\vv}{\hat{T}(\vw)}$ for 
	arbitrary $\vv,\vw \in V$. We obtain
	\begin{align*}
	\scp{\hat{T}(\vv)}{\vw} &=\scp{\sumg{g}{T_g(\vv)}}{\vw} =  \sumg{g}{\scp{T_g(\vv)}{\vw}}\\
	        &\istmit{unit.} \sumg{g}{\scp{\vv}{(T_g)\inv(\vw)}} = \sumg{g}{\scp{\vv}{(T_{g\inv})(\vw)}} \\
					&=  \sumg{g'}{\scp{\vv}{(T_{g'})(\vw)}} =  \scp{\vv}{\hat{T}(\vw)}
	\end{align*}
	\item The same proof, \emph{mutitis mutandis}, replacing $\hat{T} \leftrightarrow \hat{T^1}$,
	$T_g \leftrightarrow T_g^1$, $\vv \leftrightarrow f$, and $\vw \leftrightarrow h$ shows that
	$\scp{\hat{T^1}(f)}{h} = \scp{f}{\hat{T^1}(h)}.$
\end{enumerate}
\end{proof}

Consequently, $ \hat{T}$ and $\hat{T^1}$ are unitarily diagonalizable with real spectrum.

\begin{proposition}\label{pproph}
The operators $ \hat{T}$ and $\hat{T^1}$ are \inx{idempotent}, i.e.
\begin{enumerate}
	\item $ \hat{T} \circ  \hat{T} =  \hat{T} $ 
	\item $ \hat{T^1} \circ  \hat{T^1} =  \hat{T^1} $ .
\end{enumerate}
In particular, the eigenvalues of both operators are either $0$ or $1$.
\end{proposition}
\begin{proof}
Again, we show only one part, the other part is analog. To begin with,
let $s \in G$ be fixed. Then
	\begin{align*}
	T_s \circ \hat{T} &= T_s \circ \sumg{g}{T_g} = \sumg{g}{T_s \circ T_g} \\
	                  &= \sumg{g}{T_{sg} } = \sumg{g'}{T_{g'} } = \hat{T}. 
	\end{align*}
	From this it follows that
	\begin{align*}
	\hat{T} \circ \hat{T} &= \left (\sumg{g}{T_g} \right) \circ \hat{T}  
	                   =  \sumg{g}{T_g\circ \hat{T} } \istmit{above} 
										 =  \sumg{g}{ \hat{T} }  \\
										&= \frac{1}{|G|} \cdot |G| \cdot\hat{T} = \hat{T}.
	\end{align*}
	From $ \hat{T} \circ  \hat{T} =  \hat{T} $  we conclude that $ \hat{T} \circ  (\hat{T} - \id) =  0 $.
	Thus the minimal polynomial of $T$ divides the polynomial $\lambda (\lambda - 1)$, so
	all eigenvalues are contained in $\set{0,1}$.
\end{proof}

We will now look at the eigenvalues of $T_g$ and $T^1_g$ and their 
interrelation. Since both operators are unitary, their eigenvalues
have absolute value $1$. 

\begin{proposition}\label{pvictor}

\begin{enumerate}
	\item If $\vv \in V$ is an eigenvector of $T_g$ for the eigenvalue $\lambda$, 
	      then $\vv$ is an eigenvector of $T_{g\inv}$ for the eigenvalue $\overline{\lambda} = \frac{1}{\lambda}$. 
	\item If $f\in V^\ast$ is an eigenvector of $T^1_g$ for the eigenvalue $\lambda$, 
	      then $f$ is an eigenvector of $T^1_{g\inv}$ for the eigenvalue $\frac{1}{\lambda}$. 
	\item If $f\in V^\ast$ is an eigenvector of $T^1_g$ for the eigenvalue $\lambda$,
			  then $P(f) \in V$ is an eigenvector of $T_g$  for the eigenvalue $\overline{\lambda} = \frac{1}{\lambda}$.
	\item If $\vv \in V$ is an eigenvector of $T_g$ for the eigenvalue $\lambda$, 
				then $P\inv (\vv) \in V^\ast$  is an eigenvector of $T^1_g$ for the eigenvalue $\overline{\lambda}=\frac{1}{\lambda}$.
\end{enumerate}
\end{proposition}
\begin{proof}
We will make use of the commutativity of Proposition \ref{pprope}. Observe that $g\act\vv = T_g(\vv)$
and $g\act f = f \circ T_g$.
\begin{enumerate}
	\item \quad
	\begin{align*}
	T_g(\vv) &= g\act\vv  = \lambda \vv \implies g\inv \act g\act\vv  = g\inv \act \lambda \vv  
	             \implies g\inv \act g\act\vv  =  \lambda  g\inv \act\vv \\
					 & \implies	\vv  =  \lambda  g\inv \act\vv \implies T_{g\inv}(\vv) = g\inv \act \vv = \frac{1}{\lambda} \vv					
	\end{align*}
		\item \quad
	\begin{align*}
	T^1_g(f) &= g\act f  = \lambda f \implies g\inv \act g\act f  = g\inv \act \lambda f  
	             \implies g\inv \act g\act f  =  \lambda  g\inv \act f \\
					 & \implies	f  =  \lambda  g\inv \act f \implies T^1_{g\inv}(f) = g\inv \act f = \frac{1}{\lambda} f					
	\end{align*}
		\item \quad
	\begin{align*}
	T^1_g(f) = \lambda f &\folgtmit{P\circ} P(T^1_g(f)) = P(\lambda f) \folgtmit{(1)} T_g(P(f)) = P(\lambda f)  \\
	                     &\implies T_g(P(f)) = \overline{\lambda} P( f)  =\frac{1}{\lambda} P( f) 
	\end{align*}
		\item \quad
	\begin{align*}
	T_g(\vv)  = \lambda \vv &\folgtmit{P\inv \circ} P\inv(T_g(\vv))  = P\inv (\lambda \vv) 
	          \folgtmit{\square}  (T_g^1 \circ P\inv)(\vv) =  \overline{\lambda} P\inv (\vv) \\
						&\implies T_g^1 ( P\inv (\vv)) = \frac{1}{\lambda}  P\inv (\vv)
	\end{align*}

\end{enumerate}
\end{proof}

This implies that if we consider the union of the spectra over all $g\in G$,
then we obtain the same (multi)set, no matter if we take $T_g$ or  $T^1_g$.  \par
\newpage
\section{Eigenvectors and eigenvalues}\label{svictor}

Now we continue from where we left at the end of section \ref{spreliminaries},
fixing one group element $g \in G$ and compare  $T_g^1$ with $T_g^d$ for $d > 1$.
By a method called \kwd{stars and bars} it is easy to see that $$\tilde{d} := \dim_\setC(R_d) 
= \frac{(n+d+1)!}{(n-1)!d!} .$$

Remember that every $T_g^1$ is unitarily diagonalizable with eigenvalues of absolute value $1$.
If $\spec(T_g^1)  = (\omega_1,\dots , \omega_n) \in U(1)^n $, 
then $V^\ast$ has an orthonormal basis $\vy_g^1  := (y_{1}, \dots ,y_{n} )$, 
 such that $T_g^1 (y_{i}) = \omega_i \cdot y_{i} $ for all $1 \le i \le n$,
and $[T_g^1]_{\vy_g^1,\vy_g^1} = \diag(\omega_1,\dots , \omega_n)$.
Moreover, 
$$
[T_g^1]_{\vy_g^1,\vy_g^1} = [\id]_{\vy_g^1, \vx} \cdot [T_g^1]_{\vx,\vx} \cdot [\id]_{ \vx, \vy_g^1} 
 = \diag(\omega_1,\dots , \omega_n) ,
$$
where $[\id]_{\vy_g^1, \vx} = [\id]_{ \vx, \vy_g^1}^\ast$ is unitary. \par

For $d>1$ put 
$$
\vx^d := (x_1^d, x_2^d, \dots , x_n^d, x_1^{d-1}x_2 ,x_1^{d-1}x_3 , \dots ,x_1^{d-1}x_n, \dots )
 =: (\tilde{x_1}, \dots, \tilde{x}_{\tilde{d}}) 
,$$
all monomials in the $x_i$ of total degree $d$, numbered from $1$ to $\tilde{d}$.

These are certainly linear independent, since we have
no relations amongst the variables,  and span $R_d$, since 
every monomial of total degree $d$ can be written as a linear combination of these.
So the form a basis for $R_d$. We will not require that this can be made into an
orthonormal basis, we do not even consider any inner product on $R_d$ for $d>1$.

We rather want to establish that 
$$
\vy^d := (y_1^d, y_2^d, \dots , y_n^d, y_1^{d-1}y_2 ,y_1^{d-1}y_3 , \dots ,y_1^{d-1}y_n, \dots )
 =: (\tilde{y_1}, \dots, \tilde{y}_{\tilde{d}}) 
$$
is a basis of eigenvectors of $T_g^d$ diagonalizing $T_g^d$, using the 
same numbering.  

Arranging the eigenvalues of $T_g^1$ in the sam way we put
$$
\mathbf{\omega}^d := (\omega_1^d, \omega_2^d, \dots , \omega_n^d, \omega_1^{d-1}\omega_2 ,\omega_1^{d-1}\omega_3 , 
\dots ,\omega_1^{d-1}\omega_n, \dots )
 =: (\tilde{\omega_1}, \dots, \tilde{\omega}_{\tilde{d}}). 
$$

Now we establish that the $\tilde{y_i}$, $1\le i \le \tilde{d}$ are
the eigenvectors for the eigenvalues  $\tilde{\omega_1}$ of $T_g^d$.

\begin{proposition}\label{pinducedeigen}
In the context above,
$$
T_g^d (\tilde{y_i}) = \tilde{\omega_i} \cdot \tilde{y_i}
$$
for all $1\le i \le \tilde{d}$.
\end{proposition}
\begin{proof}
The key is proposition \ref{pindact2}, as in the preliminary observations
at the end of section \ref {spreliminaries}. Let 
$$
\tilde{y_i} = \prod_{j=1}^{n} y_j^{\epsilon_j}
$$
and
$$
\tilde{\omega_i} = \prod_{j=1}^{n} \omega_j^{\epsilon_j}
,$$
where $\epsilon_j \in \setN$ and the sum of these exponents is $d$.
Then
\begin{align*}
T_g^d (\tilde{y_i}) &= T_g^d \left ( \prod_{j=1}^{n} y_j^{\epsilon_j} \right ) 
                     =   \prod_{j=1}^{n} T_g^1 \left ( y_j^{\epsilon_j} \right )  
										 =   \prod_{j=1}^{n} \omega_j^{\epsilon_j}   y_j^{\epsilon_j} 
										 =  \tilde{\omega_i} \cdot \tilde{y_i}
\end{align*}

\end{proof}

As a consequence, $R_d$ has a basis of eigenvectors of $T_g^d$ and 
$T_g^d$ is similar to the \inx{diagonal matrix}
$\diag(\tilde{\omega_1}, \dots, \tilde{\omega}_{\tilde{d}})$.

\newpage
\section{Moliens Theorem}\label{sstart}

We will now make some final preparations and then present 
the proof of Moliens Theorem.\par

For $f \in R$ and $g \in G$ we say that $f$ is an \kwd{invariant} 
of $g$ if $g\act f = f$ and that $f$ is a (simple) invariant
of $G$ if $\forall g \in G :  g\act f = f$. 
The method of averaging from section \ref{sreynolda} can also be applied to create invariants:

\begin{proposition}\label{ppropg}
For  $f\in V^\ast$ put $\hat{f} := \hat{T^1} (f)$. Then $\hat{f}$ is an 
invariant of $G$.
\end{proposition}
\begin{proof}
Let $g \in G$ be arbitrary. We will show that $g.\hat{f} = \hat{f}$. Clearly,
from proposition \ref{pprop0} we get that
\begin{align*}
g.\hat{f} &=  \hat{f} \circ T_{g\inv} = (\hat{T^1} (f)) \circ T_{g\inv} \\
          &= \left( \sumg{s}{T_s^1(f) }\right )\circ T_{g\inv} = \left( \sumg{s}{f \circ T_{s\inv} }\right )\circ T_{g\inv}\\
					&= \sumg{s}{f \circ T_{s\inv} \circ T_{g\inv}} = \sumg{t}{f \circ T_{t\inv}} = \hat{f}.
\end{align*}
\end{proof}

Now, we call
$$
R^G := \genset{f\in R}{\forall g \in G :  g\act f = f}
$$
the \kwd{algebra of invariants} of $G$. 
 
\begin{proposition}\label{pinvalg} 
$R^G$ is a subalgebra of $R$. 
\end{proposition}
\begin{proof}
Since the mapping $f \mapsto g.f$ is linear for every $g\in G$,
$R^G$ is the intersection of subspaces, and hence a subspace.
Let us check the subring conditions in more detail.
For arbritrary  $g \in G$, $f,h \in R^G$, and $\vv \in V$ we have
$g\act f = f$, $g\act h = h$
\begin{enumerate}
	\item For the zero $0 \in R$ we obtain $(g\act 0)(\vv) = 0(g\inv\act \vv ) = 0(\vv)$,
	      so $0 \in R^G$.
	\item We see
	      \begin{align*} 
				g\act (f-h)(\vv) &= (f-h)(g\inv \act \vv) = f(g\inv \act \vv) - h(g\inv \act \vv) \\
				    &= (g \act f)(\vv) -  (g \act h)(\vv) = f(\vv) - h(\vv) = (f-h)(\vv)
				\end{align*}	
	\item Likewise,
	      \begin{align*} 
				g\act (f\cdot h)(\vv) &= (f\cdot h)(g\inv \act \vv) = f(g\inv \act \vv) \cdot  h(g\inv \act \vv) \\
				    &= (g \act f)(\vv) \cdot   (g \act h)(\vv) = f(\vv) \cdot  h(\vv) = (f\cdot h)(\vv).
				\end{align*}					
\end{enumerate}

\end{proof}

Our subalgebra $R^G$ is graded in the same way as $R$.
\begin{proposition}\label{pinvalggraded} 
The algebra of invariants of $G$ is naturally graded as
$$
R^G = \bigoplus_{d \in \setN} R^G_d,
$$ 
where $R^G_d = \genset{f\in R_d}{\forall g \in G :  g\act f = f}$,
called the $d$--th \kwd{homogeneous component} of $R^G$.
\end{proposition}
\begin{proof}
This follows directly from proposition \ref{pindact}  and proposition \ref{pindact2}.
\end{proof}

\begin{definition}[Molien series]\label{dmolien}
Viewing $R^G_d$ as a vector space, we define
$$
a_d := \dim_\setC R^G_d,
$$
the number of linearly independent homogeneous invariants of degree $d\in \setN$, and
$$
\Phi_G(\lambda) := \sum_{d\in\setN} a_d \lambda^d,
$$ 
the \kwd{Molien series} of $G$.
\end{definition}

Thus, the Molien series of $G$  is an ordinary power series generating 
function whose coefficients are  the numbers of linearly independent homogeneous invariants of degree $d$.
The following beautiful formula gives these numbers, its proof is the 
aim of this paper.

\begin{theorem}[Molien, 1897]\label{tmolien}
$$
\Phi_G(\lambda) := \frac{1}{|G|} \sum_{g\in G} \frac{1}{\det(\id - \lambda T_g)}  
$$
\end{theorem}

Following \cite{Sloane1} we first look the number $a_1$ of linearly independent homogeneous invariants of degree $d$.
\begin{theorem}[Theorem 13 in \cite{Sloane1}]\label{t13}
$$
a_1 = \trace (\hat{T}) = \trace (\hat{T^1}) 
$$
\end{theorem}
\begin{proof}
First, we note that the equation $\trace (\hat{T}) = \trace (\hat{T^1}) $ follows from
the remark at the end of section \ref{sreynolda}, since the sum for the 
trace runs over all group elements. Remember that the trace is independent 
of the choice of basis.
From proposition \ref{pproph} we know that both operators are idempotent hermitian
and  $V^\ast$ has a an orthornormal basis $\vf = (\vf_a,\dots, \vf_n)$ 
of eigenvectors of  $\hat{T^1}$, corresponding to the eigenvalues
$\lambda_1, \dots , \lambda_n \in \set{0,1}$, so
$$
[\hat{T^1}]_{\vf,\vf} = \diag(\lambda_1, \dots , \lambda_n).
$$

Let us say that this matrix has $r$ entries $1$ and the remaining $n-d$ entries $0$.
By rearranging the eigenvalues and eigenvectors we may assume that the  
first $r$ entries are $1$ and the remaining $n-d$ are $0$, i.e.
$$
\left ([\hat{T^1}]_{\vf,\vf}\right )_{i,i} = 
\begin{cases}
1 & :  1 \le i \le r\\
0 & :  r+1 \le i \le n.
\end{cases}
$$
Hence $\hat{T^1} (f_i) = f_i$ for  $1 \le i \le r$ and $\hat{T^1} (f_i) = 0$ for $r+1 \le i \le n$.
Any linear invariant of $G$ is certainly fixed by $\hat{T^1}$,
so $a_1 \le r$. On the other hand, by proposition \ref{ppropg}, 
$\hat{f_i} := \hat{T^1} (f_i) = \lambda_i f_i$ is an invariant of $G$ for every $1\le i\le r$,
so $a_1 \ge r$. Together, $a_1 = r$.
\end{proof}

Before the final proof, let us introduce a handy notation.
\begin{definition}\label{dcorfficient}
Let $p(\lambda) \in \setC[\lambda]$ or $p(\lambda) \in \setC[[\lambda]]$. Then
$[\lambda^i]:p(\lambda)$ denotes the \inx{coefficient} of $\lambda^i$ in $p(\lambda)$. 
\end{definition} 

So, for example $[x^2]: 2x^3 + 42x^2 - 6 = 42$ and $[\lambda^d]:\Phi_G(\lambda) = a_d$.
\begin{proof}{(Moliens Theorem)}
We just established the case $d = 1$, so the reader is probably 
expecting a proof by induction over $d$. But this is \emph{not} the case.
Rather, the case $d = 1$ applies to all $d  > 1$.
Note that $a_d$ is equal to the number of linearly independent
invariants of all of the $T_g^d$. So Theorem \ref{t13} gives us
\begin{align*}
a_1 &= \trace (\hat{T}) = \trace (\hat{T^1}) \qquad \mathrm{ and} \qquad\\
a_d &= \trace (\hat{T^d}),
\end{align*}
where the latter includes the first. 
From definition \ref{dreynolds} we also have
$$
 \hat{T^1} = \sumg{g}{T^1_g} 
\quad
\textrm{and in general}
\quad
 \hat{T^d} = \sumg{g}{T^d_g} ,
$$
so we already know that
$$
a_d = \sumg{g}{\trace(T^d_g)}.
$$
So all we need to show is
$$
[\lambda^d]:\frac{1}{|G|} \sum_{g\in G} \frac{1}{\det(\id - \lambda T^1_g)}  =  \sumg{g}{\trace(T^d_g)}.
$$
We will show that for every summand (group element) the equation
$$
[\lambda^d]:  \frac{1}{\det(\id - \lambda T^1_g)} = \trace(T^d_g) 
$$
holds. From proposition \ref{pinducedeigen} we get for every $g\in G$ that
\begin{align*}\trace(T^d_g) &=  \trace(\diag(\tilde{\omega_1}, \dots, \tilde{\omega}_{\tilde{d}})) \\&= 
\tilde{\omega_1} +  \dots +  \tilde{\omega}_{\tilde{d}} =
\end{align*}
sum of the products
of the $\omega_1, \omega_2, \dots ,\omega_n $, taken $d$ of them at a time.
On the other hand, for the same $g\in G$ we 
obtain from section \ref{svictor} that
 $[T_g^1]_{\vy_g^1,\vy_g^1} = \diag(\omega_1,\dots , \omega_n)$
so that

\begin{align*}
\det(\id - \lambda T^1_g) &= \det(\id - \lambda \cdot \diag(\omega_1,\dots , \omega_n) ) \\
                          &= (1 - \lambda \omega_1 )(1 - \lambda \omega_2 )\dots(1 - \lambda \omega_n ),
\end{align*}
so
\begin{align*}
\quad & \frac{1}{\det(\id - \lambda T^1_g)} = \frac{1}{(1 - \lambda \omega_1 )(1 - \lambda \omega_2 )\dots(1 - \lambda \omega_n )} \\
         &= \frac{1}{(1 - \lambda \omega_1) } \cdot  \frac{1}{(1 - \lambda \omega_2)} \cdot  \dots  \frac{1}{(1 - \lambda \omega_n)} \\
				 &= (1 + \lambda \omega_1 + \lambda^2 \omega_1^2 + \dots )(1 + \lambda \omega_2 + \lambda^2 \omega_2^2 + \dots  ) \dots
				    (1 + \lambda \omega_n + \lambda^2 \omega_n^2 + \dots ) 
\end{align*}
and here the coefficient of $\lambda^d$ is also sum of the products
of $\omega_1, \omega_2, \dots ,\omega_n $, taken $d$ of them at a time.

Again, the last claim
$$
\frac{1}{|G|} \sum_{g\in G} \frac{1}{\det(\id - \lambda T_g)} = 
\frac{1}{|G|} \sum_{g\in G} \frac{1}{\det(\id - \lambda T^1_g)}
$$
follows from
the remark at the end of section \ref{preynolds}, since the sum  runs over all group elements. 
\end{proof}

\newpage
\section{Symbol table}\label{ssymbol}

\begin{multicols}{2}
\begin{description}
	\item[$a_d$] number of linearly independent homogeneous invariants of degree $d$
	\item[$\tilde{d}$] Dimension of $R_d$
	\item[$\mcb$] ON basis for $V$
  \item[$G$] Finite group
	\item[$\omega_i$] eigenvalue of $T_g^1$ (\cite{Sloane1} $= w_i$ )
	\item[$P(f)$] \lq\lq Rho\rq\rq\ Riesz vector of $f$.
	\item[$\rho$] Unitary representation $\rho : G \to U(V), g \mapsto T_g$ 
	\item[$R$] Big algebra, direct sum of 
	\item[$R_d$] Direct summand of degree $d$
	\item[$R^G$] Ring of invariants of $d$
	\item[$R^G_d$] Degree $d$ summand 
	\item[$T_g$] representation of $g$ on $V$, (\cite{Sloane1} $ A_\alpha= [T_{g_\alpha}]_{\mcb, \mcb} $ )
	\item[$V$] Complex inner product space 
	\item[$V^\ast$] Algebraic dual of $V$
\end{description}
\end{multicols}

\section{Lost and found}\label{slostfound}

Some things to explore from here:
\begin{itemize}
	\item If we know the conjugacy classes of $G$, we may be able to say more, since every
	      unitary representation splits into irreducible components.  
	\item There seems to be a link to P\'olya enumeration.
	\item We have GAP code, see \cite{GAP4}.
	\item An example would be nice.			
	\item Relations on the generators in $S$ of the Cayley graph $\Gamma(G,S)$
		    should lead to conditions of the minimal polynomial of its adjacency operator $Q(\Gamma(G,S))$. 
	\item Also, Cayley graphs of some finite reflection groups \cite{Hu2} should become accessible.			
	\item Check some more applications, as mentioned in \cite{Sloane1}.			
	\item For finding invariants, check also \cite{Cox}, Gr\"obner bases. 
\end{itemize}

%
%

\addcontentsline{toc}{section}{Index}
\printindex
%
%
\def\thefootnote{} 
\footnote{\texttt{\jobname .tex} Typeset: \today }
%
%
\end{document}